\documentclass[12pt,oneside,english]{amsart}
\usepackage[T1]{fontenc}
\usepackage[latin9]{inputenc}
\usepackage{geometry}
\usepackage{amstext}
\usepackage{amsthm}
\usepackage{amssymb}

\makeatletter
\numberwithin{equation}{section}
\numberwithin{figure}{section}
\theoremstyle{plain}
\newtheorem{thm}{\protect\theoremname}
\theoremstyle{plain}
\newtheorem{lem}[thm]{\protect\lemmaname}

\allowdisplaybreaks

\makeatother

\usepackage{babel}
\providecommand{\lemmaname}{Lemma}
\providecommand{\theoremname}{Theorem}

\begin{document}
\title{Endpoint entropy Fefferman-Stein bounds for commutators}
\begin{abstract}
In this paper endpoint entropy Fefferman-Stein bounds for Calderón-Zygmund
operators introduced by Rahm in \cite{R} are extended to iterated
Coifman-Rochberg-Weiss commutators.
\end{abstract}

\author{Pamela A. Muller}
\email{pamela.muller@uns.edu.ar }
\address{Departamento de Mátematica. Universidad Nacional del Sur - Instituto
de Matemática de Bahía Blanca (CONICET), Bahía Blanca, Argentina}
\author{Israel P. Rivera-Ríos}
\email{israelpriverarios@uma.es}
\address{Departamento de Análisis Matemático, Estadística e Investigación Operativa
y Matemática Aplicada. Facultad de Ciencias, Universidad de Málaga,
España.}
\maketitle

\section{Introduction and main result}

In the last decade, quantitative weighted estimates have been an important
topic of study in harmonic analysis. The motivation of the results
that we present here can be traced back to the so called Muckenhoupt-Wheeden
conjecture. It is a classical result due to Fefferman and Stein that
if $w$ is any weight, namely a non negative locally integrable function,
then 
\begin{equation}
w\left(\left\{ x\in\mathbb{R}^{n}\,:\,Mf(x)>t\right\} \right)\leq c_{n}\frac{1}{t}\int_{\mathbb{R}^{n}}|f|Mw\label{eq:FS}
\end{equation}
where $c_{n}$ is a constant depending just on $n$ and $M$ stands
for the classical Hardy-Littlewood maximal function,
\[
Mf(x)=\sup_{x\in Q}\frac{1}{|Q|}\int_{Q}|f|
\]
where each $Q$ is a cube with its sides parallel to the axis. The
Muckenhoupt-Wheeden conjecture considered the posibility of replacing
$M$ by the Hilbert transform in \ref{eq:FS}. In the case of dyadic
models that conjecture was disproved by Reguera in \cite{R1} and
for the Hilbert transform by Reguera, as well, in a joint work with
Thiele \cite{RT}. 

Being that conjecture disproved a natural question would be whether
(\ref{eq:FS}) could hold for Calderón-Zygmund operators or at least
for the Hilbert transform with the maximal operator in the right hand
side replaced by a slightly larger one. That direction of research
had been already followed in the 90s by Pérez \cite{PSIO}, who showed
that the following inequality holds
\begin{equation}
w\left(\left\{ x\in\mathbb{R}^{n}\,:\,Tf(x)>t\right\} \right)\leq c_{n,\rho}\int_{\mathbb{R}^{n}}|f|M_{L(\log L)^{\rho}}w\qquad\rho>0\label{eq:FST}
\end{equation}
where $T$ stands for any Calderón-Zygmund, $c_{n,\rho}$ is a constant
that blows up when $\rho\rightarrow0$, and
\[
M_{L(\log L)^{\rho}}w=\sup_{x\in Q}\|w\|_{L(\log L)^{\rho}}.
\]
In order to make sense of $M_{L(\log L)^{\rho}}w$, we recall that,
given a Young function $A:[0,\infty)\rightarrow[0,\infty)$, namely
a convex function such that $\lim_{t\rightarrow\infty}\frac{A(t)}{t}=\infty$
and $A(0)=0$ we define 
\[
\|f\|_{A(L),Q}=\inf\left\{ \lambda>0\,:\frac{1}{|Q|}\int_{Q}A\left(\frac{|f|}{\lambda}\right)\leq1\right\} .
\]
Abusing of notation we shall denote $\|f\|_{L(\log L)^{\gamma},Q}$
in the case in which $A(t)=t\log^{\gamma}(e+t)$ and analogously,
for instance $\|f\|_{L(\log\log L)^{\gamma},Q}$, for the case $A(t)=t\log^{\gamma}(e^{e}+\log(e+t))$.
A fundamental property of these averages is that if $A(t)\leq B(t)$
for every $t\geq t_{0}$ for a certain $t_{0}\geq0$, then 
\[
\|f\|_{A(L),Q}\lesssim\|f\|_{B(L),Q}.
\]
Furthermore, they satisfy a generalized Hölder inequality. If $A,B,C$
are Young functions such that $A^{-1}(t)B^{-1}(t)\lesssim C^{-1}(t)$,
then
\[
\|fg\|_{C,Q}\lesssim\|f\|_{A,Q}\|g\|_{B,Q}.
\]
Coming back to our discussion, it is worth noting that the development
of sparse domination theory led, directly or indirectly, to several
improvements for (\ref{eq:FST}). 
\begin{itemize}
\item In \cite{HP2} it was established that $c_{n,\rho}\simeq c_{n}\frac{1}{\rho}$
in (\ref{eq:FST}). That blow up in $\rho$ is sharp, for instance,
due to the sharp dependence on the $A_{1}$ constant for the Hilbert
transform settled in \cite{LNO}.
\item In \cite{DSLR} it was settled that $M_{L(\log L)^{\rho}}$ in (\ref{eq:FST})
could be replaced for even smaller operators such as $M_{L(\log\log L)^{1+\rho}}$
keeping $c_{n,\rho}\simeq c_{n}\frac{1}{\rho}$ as well.
\item In \cite{CLO} it was shown that if
\[
\lim_{t\rightarrow\infty}\frac{\phi(t)}{t\log\log(t)}=0
\]
then (\ref{eq:FST}) with $M_{\phi}$ in place of $M_{L(\log L)^{\rho}}$
cannot hold. Up until now the whether (\ref{eq:FST}) holds with $M_{L\log\log L}$
in the right hand side remains an open question.
\end{itemize}
Quite recently another line of research related to Fefferman Stein
estimates was initiated by Rahm in \cite{R}. The new approach consisted
in replacing in (\ref{eq:FST}) $M_{L(\log L)^{\rho}}$ by a suitable
entropy bump type maximal operator encoding $A_{\infty}$ type information
of the weight. Entropy bump conditions were introduced by Treil and
Volberg \cite{TV} to obtain sufficient conditions for the two weight
boundedness of Calderón-Zygmund operators. Also in \cite{TV} it was
shown for the case $p=2$ that entropy bump conditions are slightly
more general than the bump conditions introduced by Pérez in \cite{PMax}.
An easy approach to entropy bump estimates relying upon sparse domination
results was provided by Lacey and Spencer in \cite{LS} .

Let us recall now Rahm's result. Given a weight $w$, let $\rho_{w}(Q)=\frac{1}{w(Q)}\int_{Q}M(\chi_{Q}w)$
and assume that $\varepsilon:[1,\infty)\rightarrow[1,\infty)$ an
increasing function. Then we define
\[
M_{\varepsilon}w(x)=\sup_{Q}\frac{1}{|Q|}\int_{Q}w\log_{2}\left(2+\rho_{w}(Q)\right)\varepsilon(\rho_{w}(Q)).
\]
As we mentioned above, the operator $M_{\varepsilon}$ encodes $A_{\infty}$
type information since the $A_{\infty}$ constant is defined precisely
in terms of $\rho_{w}(Q)$. To be more precise $w\in A_{\infty}$
if $[w]_{A_{\infty}}=\sup_{Q}\rho_{w}(Q)<\infty$. Rahm shows that
for this operator $M_{\varepsilon}$ 
\[
w\left(\left\{ x\in\mathbb{R}^{n}\,:\,Tf(x)>t\right\} \right)\leq c_{n}\sum_{k=1}^{\infty}\frac{1}{\varepsilon\left(2^{2^{k}}\right)}\int_{\mathbb{R}^{n}}|f|M_{\varepsilon}w.
\]

Observe that $M_{\varepsilon}$ introduces a whole new scale of maximal
oparators suitable for endpoint estimates. It is not known if $M_{\varepsilon}$
is comparable to any Orlicz maximal operator as the ones mentioned
above.

Now we turn our attention to our contribution. We recall that given
a Calderón Zygmund operator $T$ and $b\in BMO$, the iterated commutator
$T_{b}^{m}$ is defined as
\[
T_{b}^{m}f(x)=[b,T_{b}^{m-1}]f(x)
\]
where 
\[
T_{b}^{1}f(x)=[b,T]f(x)=b(x)Tf(x)-T(bf)(x)
\]
is the Coifman-Rochberg-Weiss commutator.

Endpoint Fefferman-Stein type estimates for commutators have been
explored as well. The best known result up until now is the following
\cite{LORR,IFRR}. If $w$ is an arbitrary weight and $b\in BMO$
then
\[
w\left(\left\{ x\in\mathbb{R}^{n}\,:\,T_{b}^{m}f(x)>t\right\} \right)\leq c_{n,T}\frac{1}{\rho}\int_{\mathbb{R}^{n}}\Phi\left(\frac{\|b\|_{BMO}^{m}|f|}{t}\right)M_{L(\log L)^{m}(\log L)^{1+\rho}}w\qquad\rho>0.
\]
Our purpose in this note is to explore endpoint entropy bump weighted
estimates for $T_{b}^{m}$. Before presenting our results we need
a few more definitions. As we noted above, given a weight $w$ Rahm
defines $\rho_{w}(Q)$
\[
\rho_{w}(Q)=\frac{1}{w(Q)}\int_{Q}M(\chi_{Q}w).
\]
Note that, since $\frac{1}{|Q|}\int_{Q}M(\chi_{Q}w)\simeq\|w\|_{L\log L,Q}$,
we can rephrase this condition as
\[
\rho_{1,w}(Q)=\frac{\|w\|_{L\log L,Q}}{\langle w\rangle_{Q}}
\]
in the sense that $\rho_{w}(Q)\simeq\rho_{1,w}(Q)$. Hence it is natural
to generalize such a condition as follows. Given a positive integer
$k$ we define
\[
\rho_{k,w}(Q)=\frac{\|w\|_{L(\log L)^{k},Q}}{\langle w\rangle_{Q}}.
\]
Having that notation at our disposal we can also generalize the entropy
maximal function due to Rahm as follows. Given a Young $A$, a non-negative
integer $k$ and an increasing function $\varepsilon:[1,\infty)\rightarrow[1,\infty)$,
we define
\[
M_{\varepsilon,A,k}w(x)=\sup_{x\in Q}\langle w\rangle_{A,Q}\log_{2}\left(2+\rho_{k,w}(Q)\right)\varepsilon\left(\rho_{k,w}(Q)\right).
\]
If $A(t)=t$, we shall drop the subscript $A$. On the other hand,
if besides $A(t)=t$ we have that $k=1$ as well this operator reduces
to Rahm's $M_{\varepsilon}$. 

Armed with the preceding definitions we can finally state the Theorem
of this paper.
\begin{thm}
\label{Thm:MainResult}Let $m$ be a positive integer. Let $b\in BMO$
and assume that $T$ is a Calderón-Zygmund operator and that $w$
is a weight. Then
\[
w\left(\left\{ x\in\mathbb{R}^{n}\,:\,\left|T_{b}^{m}f\right|>t\right\} \right)\leq\kappa_{\varepsilon}\int_{\mathbb{R}^{n}}\Phi_{m}\left(\frac{\|b\|_{BMO}^{m}|f|}{t}\right)M_{\varepsilon,L(\log L)^{m},m+1}w
\]
where $\kappa_{\varepsilon}=c_{n,T,m}\max\left\{ \sum_{r=0}^{\infty}\frac{1}{\varepsilon(2^{2^{r}})},1\right\} $.
\end{thm}

The remainder of the paper is devoted to the proof of this result.

\section{Proof of the main result}

Our proof relies upon the sparse domination result that was settled
in \cite{LORR,IFRR}.
\begin{thm}
Let $b\in L_{\text{loc}}^{1}$ and let $T$ be a Calderón-Zygmund
operator. Then there exist $N_{\alpha}$ $\alpha$-Carleson families
$\mathcal{S}_{j}$ contained in $3^{n}$ dyadic lattices such that
\[
|T_{b}^{m}f(x)|\leq c_{n,mT}\sum_{j=1}^{N_{\alpha}}\sum_{h=0}^{m}\mathcal{T}_{b,\mathcal{S}_{j}}^{h,m}f(x)
\]
where 
\begin{align*}
\mathcal{T}_{b,\mathcal{S}}^{h,m}f(x) & =\sum_{Q\in\mathcal{S}}|b(x)-b_{Q}|^{h}\frac{1}{|Q|}\int_{Q}|b-b_{Q}|^{m-h}f.
\end{align*}
\end{thm}

Observe that, in fact, it suffices to study $\mathcal{T}_{b,\mathcal{S}}^{m,m}$
and $\mathcal{T}_{b,\mathcal{S}}^{0,m}$, since, as it was shown in
\cite[Lemma 2.2]{CUMT},
\[
\mathcal{T}_{b,\mathcal{S}}^{h}f(x)\leq\mathcal{T}_{b,\mathcal{S}}^{m,m}f(x)+\mathcal{T}_{b,\mathcal{S}}^{0,m}f(x)
\]
for every $h\in\{0,\dots,m\}$.

Hence the proof of Theorem \ref{Thm:MainResult} boils down to obtaining
estimates just for $\mathcal{T}_{b,\mathcal{S}_{j}}^{m,m}$ and $\mathcal{T}_{b,\mathcal{S}_{j}}^{0,m}$.

For $\mathcal{T}_{b,\mathcal{S}}^{m,m}$ we provide the following
result.
\begin{thm}
\label{Thm:TmbS}Let $\mathcal{S}$ be a $\alpha$-Carleson family
with $0<56^{m}(\alpha-1)<1$ and $b\in BMO$. Then, 
\[
\|\mathcal{T}_{b,\mathcal{S}}^{m,m}f\|_{L^{1,\infty}(w)}\leq c_{n,\alpha}\|b\|_{BMO}^{m}\sum_{r=0}^{\infty}\frac{1}{\varepsilon(2^{2^{r}})}\|f\|_{L^{1}(M_{\varepsilon,L(\log L)^{m},m+1}w)}
\]
where $\varepsilon:[1,\infty)\rightarrow[1,\infty)$ is an increasing
function.
\end{thm}

Observe that for $\mathcal{T}_{b,\mathcal{S}}^{0,m}$ the following
estimate can be recovered from the arguments in \cite{LORR,IFRR}
\begin{thm}
Let $\mathcal{S}$ be a Carleson family and let $b\in BMO$. Then
\[
w\left(\left\{ x\in\mathbb{R}^{n}\,:\,\left|\mathcal{T}_{b,\mathcal{S}}^{0,m}f\right|>t\right\} \right)\leq c_{n,m,\alpha}\int_{\mathbb{R}^{n}}\Phi_{m}\left(\frac{\|b\|_{BMO}^{m}|f|}{t}\right)M_{L(\log L)^{m}}w
\]
where $\Phi_{m}(t)=t\log^{m}(e+t)$.
\end{thm}

Since $M_{L(\log L)^{m}}w\leq M_{\varepsilon,L(\log L)^{m},m+1}w$
that estimate is good enough for our us. The main theorem readily
follows from the combination of the results above, hence it suffices
to settle Theorem \ref{Thm:TmbS} to end the proof. We devote the
remainder of the section and of the paper to that purpose.

\subsection{Lemmatta }

Arguing as in \cite[6.6 Lemma]{HP1} we can get the following lemma.
\begin{lem}
\label{lem:BorrowHP}For a cube $Q$ and a subset $E\subsetneq Q$
we have that
\[
w(E)\leq\frac{2^{n+3}\rho_{w}(Q)}{\log\left(\frac{|Q|}{|E|}\right)}w(Q).
\]
\end{lem}

Lemma \ref{lem:BorrowHP} is an important tool in \cite{R}. In the
following lines we present a result generalizes the lemma above. Before
that we recall that it is a well known fact that
\[
\|w\|_{L\log L^{k},Q}\simeq\frac{1}{|Q|}\int_{Q}w\log\left(e+\frac{w}{w_{Q}}\right)^{k}
\]
and it is also well known that for some $\kappa_{k}\geq1$ 
\[
\Phi_{k}(ab)\leq\kappa_{k}\Phi_{k}(a)\Phi_{k}(b)
\]
where $\Phi_{k}(t)=t\log^{k}(e+t)$. Bearing those facts in mind we
can settle the following Lemma.
\begin{lem}
\label{lem:LetGenHP}Let $Q$ be a cube and $E\subsetneq Q$. Then
there exists $c>0$ depending just on $k$ and $n$, such that
\[
\|w\chi_{E}\|_{L\log^{k}L,Q}\leq c\frac{\log\left(e+\log\left(\frac{|Q|}{|E|}\right)\right)^{k}}{\log\left(\frac{|Q|}{|E|}\right)}\|w\|_{L\log^{k+1}L,Q}
\]
and consequently 
\[
\|w\chi_{E}\|_{L\log^{k}L,Q}\leq c\frac{\log\left(e+\log\left(\frac{|Q|}{|E|}\right)\right)^{k}}{\log\left(\frac{|Q|}{|E|}\right)}\rho_{k+1,w}(Q)\langle w\rangle_{Q}.
\]
\end{lem}

\begin{proof}
Let $J_{\gamma}=\left\{ x\in Q\,:\,w(x)>e^{\gamma}\langle w\rangle_{Q}\right\} $.
First we observe that
\begin{align*}
\frac{1}{|Q|}\int_{J_{\gamma}}\frac{w}{\lambda_{0}\langle w\rangle_{Q}}\log^{k}\left(e+\frac{w}{\lambda_{0}\langle w\rangle_{Q}}\right) & \leq\frac{1}{|Q|}\int_{J_{\gamma}}\frac{w}{\lambda_{0}\langle w\rangle_{Q}}\log^{k}\left(e+\frac{w}{\lambda_{0}\langle w\rangle_{Q}}\right)\frac{\log\left(e+\frac{w}{\langle w\rangle_{Q}}\right)}{\log\left(e+\frac{w}{\langle w\rangle_{Q}}\right)}\\
 & \leq\frac{1}{\gamma}\frac{1}{|Q|}\int_{Q}\frac{w}{\lambda_{0}\langle w\rangle_{Q}}\log^{k}\left(e+\frac{w}{\lambda_{0}\langle w\rangle_{Q}}\right)\log\left(e+\frac{w}{\langle w\rangle_{Q}}\right)\\
 & \frac{1}{\gamma}\frac{1}{|Q|}\int_{Q}\frac{w}{\lambda_{0}\langle w\rangle_{Q}}\log^{k}\left(e+\frac{w}{\lambda_{0}\langle w\rangle_{Q}}\right)\log\left(e+\frac{w}{\langle w\rangle_{Q}}\right)\\
\leq & \frac{1}{\gamma}\frac{1}{|Q|}\int_{Q}\frac{w}{\lambda_{0}\langle w\rangle_{Q}}\log^{k}\left(e+\frac{w}{\lambda_{0}\langle w\rangle_{Q}}\right)\log\left(e+\frac{w}{\langle w\rangle_{Q}}\right)\\
\leq & \frac{c_{n}\kappa_{k}}{\gamma}\Phi_{k}\left(\frac{1}{\lambda_{0}}\right)\frac{1}{\langle w\rangle_{Q}}\|w\|_{L\log^{k+1}L,Q}
\end{align*}
Consequently we have that
\begin{align*}
\frac{c_{n}\kappa}{\gamma}\Phi_{k} & \left(\frac{1}{\lambda_{0}}\right)\frac{1}{\langle w\rangle_{Q}}\|w\|_{L\log^{k+1}L,Q}\leq1\\
\iff & \Phi_{k}\left(\frac{1}{\lambda_{0}}\right)\leq\frac{\langle w\rangle_{Q}\gamma}{c_{n}\kappa\|w\|_{L\log^{k+1}L,Q}}\\
\iff & \frac{1}{\lambda_{0}}\leq\Phi_{k}^{-1}\left(\frac{\langle w\rangle_{Q}\gamma}{c_{n}\kappa\|w\|_{L\log^{k+1}L,Q}}\right)\\
\iff & \frac{1}{\Phi_{k}^{-1}\left(\frac{\langle w\rangle_{Q}\gamma}{c_{n}\kappa\|w\|_{L\log^{k+1}L,Q}}\right)}\leq\lambda_{0}
\end{align*}
and this yields
\begin{align*}
\|w\chi_{J_{\gamma}}\|_{L\log^{k}L,Q} & \leq\frac{\langle w\rangle_{Q}}{\Phi_{k}^{-1}\left(\frac{\langle w\rangle_{Q}\gamma}{c_{n}\kappa\|w\|_{L\log^{k+1}L,Q}}\right)}\simeq\frac{\langle w\rangle_{Q}\log\left(e+\frac{\langle w\rangle_{Q}\gamma}{c_{n}\kappa\|w\|_{L\log^{k+1}L,Q}}\right)}{\frac{\langle w\rangle_{Q}\gamma}{c_{n}\kappa\|w\|_{L\log^{k+1}L,Q}}}\\
 & \leq c_{n}\kappa\frac{\log^{k}\left(e+\gamma\right)}{\gamma}\|w\|_{L\log^{k+1}L,Q}
\end{align*}
Having that estimate at our disposal now we can proceed as follows.
Let 
\[
\frac{|E|}{|Q|}=e^{-\lambda}.
\]
Then
\begin{align*}
\|w\chi_{E}\|_{L\log^{k}L,Q} & \leq\|w\chi_{E\cap J_{\lambda/2}}\|_{L\log^{k}L,Q}+\|w\chi_{E\setminus J_{\lambda/2}}\|_{L\log^{k}L,Q}\\
 & \leq c_{n}\kappa\frac{\log^{k}\left(e+\lambda/2\right)}{\lambda/2}\|w\|_{L\log^{k+1}L,Q}+e^{\frac{\lambda}{2}}\frac{w(Q)}{|Q|}\frac{1}{\Phi_{k}^{-1}\left(\frac{|Q|}{|E|}\right)}\\
 & \leq2c_{n}\kappa\frac{\log^{k}\left(e+\lambda\right)}{\lambda}\|w\|_{L\log^{k+1}L,Q}+c_{n}e^{\frac{\lambda}{2}}\frac{w(Q)}{|Q|}\frac{\log^{k}(e+\frac{|Q|}{|E|})}{\frac{|Q|}{|E|}}\\
 & =2c_{n}\kappa\frac{\log^{k}\left(e+\lambda\right)}{\lambda}\|w\|_{L\log^{k+1}L,Q}+c_{n}e^{\frac{\lambda}{2}}\frac{w(Q)}{|Q|}\frac{\log^{k}(e+e^{\lambda})}{e^{\lambda}}\\
 & \leq2c_{n}\kappa\frac{\log^{k}\left(e+\lambda\right)}{\lambda}\|w\|_{L\log^{k+1}L,Q}+c_{n}e^{-\frac{\lambda}{2}}\frac{w(Q)}{|Q|}\log^{k}(e+e^{\lambda})
\end{align*}
Since we have that the first term is larger we are done.
\end{proof}
We end this section recalling that the John-Nirenberg inequality (see
for instance \cite[p. 124]{G}) tells us that if $b\in BMO$ then
\[
|\{x\in Q\,:\,|b(x)-b_{Q}|>\lambda\}|\leq e|Q|e^{-\frac{\lambda}{e2^{n}\|b\|_{BMO}}}.
\]

\subsection{Proof of Theorem \ref{Thm:TmbS}}

We shall assume that $\|b\|_{BMO}=1$ by homogeneity and also that
$f\geq0$ since $T_{b,\mathcal{S}}^{m,m}$ is a positive operartor.
Observe that we can split the sparse family as $\mathcal{S}=\mathcal{S}_{1}\cup\mathcal{\mathcal{S}}_{2}$
where $\mathcal{S}_{1}$ contains the cubes for which $1\leq\rho_{m+1,w}(Q)<2$
and $\mathcal{S}_{2}$ the remaining ones. Then
\[
w(\{T_{b,\mathcal{S}}^{m,m}f>t\})\leq w\left(\left\{ T_{b,\mathcal{S}_{1}}^{m,m}f>\frac{t}{2}\right\} \right)+w\left(\left\{ T_{b,\mathcal{S}_{2}}^{m,m}f>\frac{t}{2}\right\} \right).
\]
Observe that for the first term we have that $w\in A_{\infty}$ with
respect to the family $\mathcal{S}_{1}$ and hence, arguments in \cite{IFRR,LORR}
show that 
\[
w\left(\left\{ T_{b,\mathcal{S}_{1}}^{m,m}f>t\right\} \right)\lesssim\frac{1}{t}\int_{\mathbb{R}^{n}}fMw.
\]
However we will provide an argument for that term as well the sake
of completeness. We shall deal with those terms separatedly. We will
be done provided we can show that
\[
w\left(\left\{ T_{b,\mathcal{S}_{i}}^{m,m}f>t\right\} \right)\lesssim\frac{1}{t}\int_{\mathbb{R}^{n}}|f|M_{\varepsilon,L(\log L)^{m},m+1}w.
\]
We shall proceed as follows. First recall that by homogeneity it suffices
to show that for some $t_{0}>0$ 
\[
w(\{T_{b,\mathcal{S}_{i}}^{m,m}f>t_{0}\})\lesssim\int_{\mathbb{R}^{n}}fM_{\varepsilon,L(\log L)^{m},m+1}w.
\]
We will argue as follows for both terms. Let $\tau>0$ such that $\varphi(t)=\frac{\log^{m}(e+\log(t))}{\log^{m}(t)}$
is decreasing for $t\geq e^{4^{1+\tau}-1}$. Observe that

\begin{align*}
 & w(\{T_{b,\mathcal{S}_{i}}^{m,m}f>4^{\tau m}2^{nm}e^{m}\cdot100\})\\
 & =w\left(\left\{ T_{b,\mathcal{S}_{i}}^{m,m}f>4^{\tau m}2^{nm}e^{m}\cdot100,Mf\leq\frac{1}{56^{m}}\right\} \cup\left\{ T_{b,\mathcal{S}_{i}}^{m,m}f>4^{\alpha m}2^{nm}e^{m}\cdot100,Mf>\frac{1}{56^{m}}\right\} \right)\\
 & =w\left(\left\{ T_{b,\mathcal{S}_{i}}^{m,m}f>4^{\tau m}2^{nm}e^{m}\cdot100,Mf\leq\frac{1}{56^{m}}\right\} \right)+w\left(\left\{ Mf>\frac{1}{56^{m}}\right\} \right)\\
 & \leq w\left(\left\{ T_{b,\mathcal{S}_{i}}^{m,m}f>4^{\tau m}2^{nm}e^{m}\cdot100,Mf\leq\frac{1}{56^{m}}\right\} \right)+56^{m}\int_{\mathbb{R}^{n}}fMw.
\end{align*}
This reduces us to provide a suitable estimate for the first term.
Let us call 
\[
G_{i}=\left\{ T_{b,\mathcal{S}_{i}}^{m,m}f>4^{\tau m}2^{nm}e^{m}\cdot100,Mf\leq\frac{1}{56^{m}}\right\} .
\]
We shall assume that $w(G_{i})<\infty$ since otherwise we already
had that $w(\{T_{b,\mathcal{S}_{i}}^{m,m}f>4^{\tau m}2^{nm}e^{m}\cdot100\})=\infty$
and hence the estimate was trivial. Hence it will suffice to show
that 
\begin{equation}
w(G_{i})\leq c\int_{\mathbb{R}^{n}}fM_{\varepsilon,L(\log L)^{m},m+1}w+\nu_{i}w(G)\label{eq:Target-1}
\end{equation}
for some $\nu_{i}\in(0,1)$ in both cases. We devote the remainder
of the subsection to that purpose.

\subsubsection{Bound for $T_{b,\mathcal{S}_{1}}^{m,m}$}

We shall drop the subscripts of $\mathcal{S}_{1}$ and $G_{1}$ for
the sake of clarity. We split the family $\mathcal{S}$ as follows
$Q\in\mathcal{S}_{k}$ if and only if

\[
\frac{1}{56^{m(k+1)}}<\frac{1}{|Q|}\int_{Q}|f|\leq\frac{1}{56^{km}}.
\]
Then, $\mathcal{S}=\bigcup_{k=1}^{\infty}\mathcal{S}_{k}$. We recall,
as well, that $1\leq\rho_{m+1,w}(Q)<2$ for every cube $Q\in\mathcal{S}$.

Observe that then
\begin{equation}
w(G)\leq\frac{1}{4^{\tau m}2^{nm}e^{m}\cdot100}\sum_{k=1}^{\infty}\sum_{Q\in\mathcal{S}_{k}}\frac{1}{|Q|}\int_{Q}|f|\int_{G\cap Q}|b-b_{Q}|^{m}w\label{eq:Bound-1-1}
\end{equation}
and let us consider, as above, for $Q\in\mathcal{S}_{k}$, 
\[
F_{k}(Q)=\left\{ x\in Q\,:\,|b-b_{Q}|^{m}>2^{nm}e^{m}4^{m(k+\tau)}\right\} .
\]
Again, by John-Nirenberg theorem, since $b\in BMO$, 

\[
\frac{|F_{k}(Q)|}{|Q|}\leq ee^{-4^{k+\tau}}.
\]
Now we argue as follows. Observe that

\begin{align*}
 & \sum_{k=1}^{\infty}\sum_{Q\in\mathcal{S}_{k}}\frac{1}{|Q|}\int_{Q}|f|\int_{G\cap Q}|b-b_{Q}|^{m}w\\
 & \leq\sum_{k=1}^{\infty}\sum_{Q\in\mathcal{S}_{k}}\frac{1}{|Q|}\int_{Q}|f|\int_{G\cap F_{k}(Q)}|b-b_{Q}|^{m}w\\
 & +\sum_{k=1}^{\infty}\sum_{Q\in\mathcal{S}_{k}}\frac{1}{|Q|}\int_{Q}|f|\int_{G\cap(Q\setminus F_{k}(Q))}|b-b_{Q}|^{m}w\\
 & \leq\sum_{k=1}^{\infty}\sum_{Q\in\mathcal{S}_{k}}\frac{1}{|Q|}\int_{Q}|f|\int_{G\cap F_{k}(Q)}|b-b_{Q}|^{m}w\\
 & +2^{nm}e^{m}\cdot4^{\tau m}\sum_{k=1}^{\infty}4^{km}\sum_{Q\in\mathcal{S}_{k}}\frac{1}{|Q|}\int_{Q}|f|\int_{G\cap Q}w\\
 & =\left(L_{1}+4^{nm}e^{m}\cdot4^{\tau m}L_{2}\right)
\end{align*}

Observe that if $Q\in\mathcal{S}_{k}$ and we denote $E_{Q}=Q\setminus\bigcup_{Q'\subsetneq Q,\,Q'\in\mathcal{S}_{k}}Q'$
then 
\begin{equation}
\int_{Q}f\lesssim\int_{E_{Q}}f.\label{eq:DisjointF-1}
\end{equation}
Indeed 

\begin{align*}
\int_{Q}f & =\int_{E_{Q}}f+\sum_{\underset{Q'\in S_{k}}{Q'\subset Q}}\int_{Q'}f\\
 & \leq\int_{E_{Q}}f+\sum_{\underset{Q'\in S_{k}}{Q'\subset Q}}\frac{1}{56^{km}}|Q'|\\
 & \leq\int_{E_{Q}}f+56^{m}(\alpha-1)\frac{|Q|}{56^{(k+1)m}}\\
 & \leq\int_{E_{Q}}f+56^{m}(\alpha-1)\int_{Q}f
\end{align*}
and since $56^{m}(\alpha-1)$ we arrive to the desired conclusion. 

First we deal with $L_{1}$. Since $\varphi(t)=\frac{\log^{m}(e+\log(t))}{\log^{m}(t)}$
is decreasing for $t\geq e^{4^{1+\tau}-1}$ taking into account, that
$\frac{|F_{k}(Q)|}{|Q|}\leq ee^{-4^{k+\tau}}\iff\frac{|Q|}{|F_{k}(Q)|}\geq e^{4^{k+\tau}-1}$,
we have that by Lemma \ref{lem:LetGenHP},
\begin{align*}
\|w\chi_{F_{k}(Q)}\|_{L\log L^{m},Q} & \leq c\frac{\log\left(e+\log\left(\frac{|Q|}{|F_{k}(Q)|}\right)\right)^{m}}{\log\left(\frac{|Q|}{|F_{k}(Q)|}\right)}\rho_{m+1,w}(Q)\langle w\rangle_{Q}\\
 & \leq c\frac{\log\left(e+\log\left(e^{4^{k+\tau}-1}\right)\right)^{m}}{\log\left(e^{4^{k+\tau}-1}\right)}\rho_{m+1,w}(Q)\langle w\rangle_{Q}\\
 & \lesssim\frac{k^{m}}{4^{k}}\rho_{m+1,w}(Q)\langle w\rangle_{Q},
\end{align*}
namely
\begin{equation}
\|w\chi_{F_{k}(Q)}\|_{L\log L^{m},Q}\lesssim\frac{k^{m}}{4^{k}}\rho_{m+1,w}(Q)\langle w\rangle_{Q}.\label{eq:JNDecay}
\end{equation}
Then, since $\rho_{m+1,w}(Q)\leq2$ for every $Q\in\mathcal{S}$,
we have that

\begin{align*}
 & \sum_{k=1}^{\infty}\sum_{Q\in\mathcal{S}_{k}}\frac{1}{|Q|}\int_{Q}|f|\int_{G\cap F_{k}(Q)}|b-b_{Q}|^{m}w\\
 & \lesssim\sum_{k=1}^{\infty}\sum_{Q\in\mathcal{S}_{k}}\int_{E_{Q}}|f|\|w\chi_{F_{k}(Q)}\|_{L\log L^{m},Q}\\
 & \lesssim\sum_{k=1}^{\infty}\sum_{Q\in\mathcal{S}_{k}}\int_{E_{Q}}|f|\frac{k^{m}}{4^{k}}\rho_{m+1,w}(Q)\langle w\rangle_{Q}\\
 & \lesssim\sum_{k=1}^{\infty}\sum_{Q\in\mathcal{S}_{k}}\int_{E_{Q}}|f|\frac{2k^{m}}{4^{k}}\langle w\rangle_{Q}\\
 & \lesssim\sum_{k=1}^{\infty}\sum_{Q\in\mathcal{S}_{k}}\int_{E_{Q}}|f|\frac{2\cdot2^{k}}{4^{k}}\langle w\rangle_{Q}\\
 & \simeq\sum_{k=1}^{\infty}\frac{1}{2^{k}.}\sum_{Q\in\mathcal{S}_{k}}\int_{E_{Q}}|f|Mw\\
 & \lesssim\int_{\mathbb{R}^{n}}|f|Mw.
\end{align*}
Now we turn our attention to $L_{2}$. We begin discussing a suitable
way to break into pices a cube $Q\in\mathcal{S}_{k}$. We shall split
$\mathcal{S}_{k}^{\nu}$ where $\mathcal{S}_{k}^{0}$ is the family
of maximal cubes in $\mathcal{S}_{k}$, $\mathcal{S}_{k}^{j+1}$ is
the family of maximal cubes contained in cubes of $\mathcal{S}_{k}^{j}$
and so on. Let $Q\in\mathcal{S}_{k}^{j}$. Note that by the $\alpha$-Carleson
condition 
\[
\sum_{Q'\in\mathcal{S}_{k}^{j+1}(Q)}|Q'|\leq(\alpha-1)|Q|.
\]
where $\mathcal{S}_{k}^{j+1}(Q)$ stands for the family of cubes of
$\mathcal{S}_{k}^{j+1}$ contained in $Q$. Furthermore, iterating
the left hand side, 
\[
\sum_{Q'\in\mathcal{S}_{k}^{j+t}(Q)}|Q'|\leq(\alpha-1)^{t}|Q|.
\]
Let us call $Q^{t}=\cup_{Q'\in\mathcal{S}_{r,k}^{j+t}(Q)}Q'$. Then
we have that 
\[
Q=Q^{t}\cup\tilde{E}_{Q}
\]
where 
\[
\tilde{E}_{Q}=\bigcup_{s=1}^{t}Q\setminus\cup_{P\in\mathcal{S}_{k}^{j+s}(Q)}P.
\]
Note that for this choice of $\tilde{E_{Q}}$, 
\[
\sum_{Q\in\mathcal{S}_{k}}\chi_{\tilde{E_{Q}}}(x)\leq t.
\]
Let us choose $t=7^{km}$ . Observe that, then

\[
\frac{|Q|}{|Q^{t}|}\geq\frac{1}{\left(\alpha-1\right)^{7^{km}}}=\left(\frac{1}{\alpha-1}\right)^{7^{km}}.
\]
and since $1<\alpha<2,$
\[
\log\left(2\frac{|Q|}{|Q^{t}|}\right)\geq7^{km}\log\left(\frac{1}{\alpha-1}\right).
\]
 Having the discussion above at our disposal we now provide our estimate
for $L_{2}$. First we split the sum in two terms
\[
L_{2}=\sum_{k=1}^{\infty}4^{km}\sum_{Q\in\mathcal{S}_{k}}\frac{1}{|Q|}\int_{Q}|f|w(G\cap Q^{t})+\sum_{k=1}^{\infty}4^{km}\sum_{Q\in\mathcal{S}_{k}}\frac{1}{|Q|}\int_{Q}|f|w(G\cap\tilde{E}_{Q}).
\]
For the first term we observe that  taking into account that for
every cube $Q$, $\rho_{m+1,w}(Q)\leq2$
\begin{align*}
\frac{1}{|Q|}\int_{Q}|f|w(G\cap Q^{t}) & \leq\frac{1}{|Q|}\int_{Q}|f|w(Q^{t})\leq2\|w\|_{L\log L,Q}\|\chi_{Q^{t}}\|_{\exp(L^{m}),Q}\int_{Q}|f|\\
 & =\frac{1}{\log\left(2\frac{|Q|}{|Q^{t}|}\right)}\int_{Q}f\|w\|_{L\log L,Q}\lesssim\frac{1}{\log\left(2\frac{|Q|}{|Q^{t}|}\right)}\int_{Q}f\|w\|_{L\log L,Q}\\
 & \lesssim\frac{1}{7^{km}}\int_{E_{Q}}|f|Mw.
\end{align*}
Hence, 
\begin{align*}
\sum_{k=1}^{\infty}4^{km}\sum_{Q\in\mathcal{S}_{k}}\frac{1}{|Q|}\int_{Q}|f|w(G\cap Q^{t}) & \lesssim\sum_{k=1}^{\infty}\frac{4^{km}}{7^{km}}\sum_{Q\in\mathcal{S}_{k}}\int_{E_{Q}}|f|Mw\\
 & \lesssim\int_{\mathbb{R}^{d}}|f|Mw\sum_{k=1}^{\infty}\left(\frac{4}{7}\right)^{km}\\
 & \lesssim\int_{\mathbb{R}^{d}}|f|Mw.
\end{align*}
For the remaining term
\begin{align*}
 & \sum_{k=1}^{\infty}4^{mk}\sum_{Q\in\mathcal{S}_{k}}\frac{1}{|Q|}\int_{Q}|f|w(G\cap\tilde{E}_{Q})\\
 & =\sum_{k=1}^{\infty}4^{mk}\sum_{\nu=0}^{7^{km}}\sum_{Q\in\mathcal{S}_{k}}\sum_{Q'\in\mathcal{S}_{k}^{\nu}(Q)}\frac{1}{|Q|}\int_{Q}|f|w(G\cap Q')\\
 & \leq\sum_{k=1}^{\infty}\frac{4^{mk}}{56^{mk}}\sum_{\nu=0}^{7^{km}}\sum_{Q\in\mathcal{S}_{k}}\sum_{Q'\in\mathcal{S}_{k}^{\nu}(Q)}w(G\cap Q')\\
 & \leq\sum_{k=1}^{\infty}\frac{4^{km}}{56^{km}}7^{km}w(G)\\
 & \leq\sum_{k=1}^{\infty}\frac{1}{2^{km}}w(G)\leq w(G)
\end{align*}
and hence we are done.

\subsubsection{Bound for $T_{b,\mathcal{S}_{2}}^{m,m}$}

Again, we shall drop the subscripts of $\mathcal{S}_{2}$ and $G_{2}$
for the sake of clarity. First we split the sparse family $\mathcal{S}$
as follows. $Q\in\mathcal{S}_{r,k}$ if
\[
2^{2^{r}}\leq\rho_{m+1,w}(Q)<2^{2^{r+1}}
\]
and 
\[
\frac{1}{56^{m(k+1)}}<\frac{1}{|Q|}\int_{Q}|f|\leq\frac{1}{56^{km}}.
\]
Then, $\mathcal{S}=\bigcup_{r=0}^{\infty}\bigcup_{k=1}^{\infty}\mathcal{S}_{r,k}$.
Observe that 
\begin{equation}
w(G)\leq\frac{1}{4^{\tau m}2^{nm}e^{m}\cdot100}\sum_{r=0}^{\infty}\sum_{k=1}^{\infty}\sum_{Q\in\mathcal{S}_{r,k}}\frac{1}{|Q|}\int_{Q}|f|\int_{G\cap Q}|b-b_{Q}|^{m}w\label{eq:Bound-1}
\end{equation}
Now we further consider for $Q\in\mathcal{S}_{r,k}$
\[
F_{k}(Q)=\left\{ x\in Q\,:\,|b-b_{Q}|^{m}>2^{nm}e^{m}4^{m(k+\tau)}\right\} .
\]
Note that due to the John-Nirenberg inequality this yields
\[
\frac{|F_{k}(Q)|}{|Q|}\leq ee^{-4^{k+\tau}}.
\]
Then 
\begin{align*}
 & \sum_{r=0}^{\infty}\sum_{k=1}^{\infty}\sum_{Q\in\mathcal{S}_{r,k}}\frac{1}{|Q|}\int_{Q}|f|\int_{G\cap Q}|b-b_{Q}|^{m}w\\
 & \leq\sum_{r=0}^{\infty}\sum_{k=1}^{\infty}\sum_{Q\in\mathcal{S}_{r,k}}\frac{1}{|Q|}\int_{Q}|f|\int_{G\cap F_{k}(Q)}|b-b_{Q}|^{m}w\\
 & +\sum_{r=0}^{\infty}\sum_{k=1}^{\infty}\sum_{Q\in\mathcal{S}_{r,k}}\frac{1}{|Q|}\int_{Q}|f|\int_{G\cap(Q\setminus F_{k}(Q))}|b-b_{Q}|^{m}w\\
 & \leq\sum_{r=0}^{\infty}\sum_{k=1}^{\infty}\sum_{Q\in\mathcal{S}_{r,k}}\frac{1}{|Q|}\int_{Q}|f|\int_{G\cap F_{k}(Q)}|b-b_{Q}|^{m}w\\
 & +2^{nm}e^{m}\cdot4^{\tau m}\sum_{r=0}^{\infty}\sum_{k=1}^{\infty}4^{km}\sum_{Q\in\mathcal{S}_{r,k}}\frac{1}{|Q|}\int_{Q}|f|\int_{G\cap Q}w\\
 & =\left(L_{1}+2^{nm}e^{m}\cdot4^{\tau m}\cdot L_{2}\right)
\end{align*}

We observe that if $Q\in\mathcal{S}_{r,k}$ then 
\begin{equation}
\int_{Q}f\lesssim\int_{E_{Q}}f\label{eq:DisjointF}
\end{equation}
where $E_{Q}=Q\setminus\bigcup_{Q'\subsetneq Q,\,Q'\in\mathcal{S}_{r,k}}Q'$.
Note that it suffices to argue as we did to derive \ref{eq:DisjointF-1}
since we only used information relative to the splitting in $k$.

Let us deal now with $L_{1}$. We split the sum in $k$ as follows
\begin{align*}
L_{1} & =\sum_{r=0}^{\infty}\sum_{k=1}^{\log_{2}(2^{2^{r+1}})}\sum_{Q\in\mathcal{S}_{r,k}}\frac{1}{|Q|}\int_{Q}|f|\int_{G\cap F_{k}(Q)}|b-b_{Q}|^{m}w\\
 & +\sum_{r=0}^{\infty}\sum_{k=\log_{2}(2^{2^{r+1}})}^{\infty}\sum_{Q\in\mathcal{S}_{r,k}}\frac{1}{|Q|}\int_{Q}|f|\int_{G\cap F_{k}(Q)}|b-b_{Q}|^{m}w\\
 & =L_{11}+L_{12}.
\end{align*}
Let us focus first on $L_{11}$. Observe that 
\begin{align*}
 & \sum_{r=0}^{\infty}\sum_{k=1}^{\log_{2}(2^{2^{r+1}})}\sum_{Q\in\mathcal{S}_{r,k}}\frac{1}{|Q|}\int_{Q}|f|\int_{G\cap F_{k}(Q)}|b-b_{Q}|^{m}w\\
 & \lesssim\sum_{r=0}^{\infty}\sum_{k=1}^{\log_{2}(2^{2^{r+1}})}\sum_{Q\in\mathcal{S}_{r,k}}\int_{E_{Q}}|f|\frac{\int_{G\cap F_{k}(Q)}|b-b_{Q}|^{m}w}{|Q|}\\
 & \lesssim\sum_{r=0}^{\infty}\sum_{k=1}^{\log_{2}(2^{2^{r+1}})}\sum_{Q\in\mathcal{S}_{r,k}}\int_{E_{Q}}|f|\|w\chi_{F_{k}(Q)}\|_{L\log L^{m},Q}\||b-b_{Q}|^{m}\|_{\exp L^{\frac{1}{m}},Q}\\
 & \lesssim\sum_{r=0}^{\infty}\sum_{k=1}^{\log_{2}(2^{2^{r+1}})}\sum_{Q\in\mathcal{S}_{r,k}}\frac{\log_{2}\left(2+\rho_{m+1,w}(Q)\right)\varepsilon\left(\rho_{m+1,w}(Q)\right)}{\log_{2}\left(2+\rho_{m+1,w}(Q)\right)\varepsilon\left(\rho_{m+1,w}(Q)\right)}\int_{E_{Q}}|f|\|w\chi_{F_{k}(Q)}\|_{L\log L^{m},Q}\\
 & \lesssim\sum_{r=0}^{\infty}\sum_{k=1}^{\log_{2}(2^{2^{r+1}})}\frac{1}{\log_{2}(2+2^{2^{r}})\varepsilon(2^{2^{r}})}\sum_{Q\in\mathcal{S}_{r,k}}\int_{E_{Q}}|f|M_{\varepsilon,L\log L,^{m}}w\\
 & \lesssim\sum_{r=0}^{\infty}\sum_{k=1}^{\log_{2}(2^{2^{r+1}})}\frac{1}{\log_{2}(2+2^{2^{r}})\varepsilon(2^{2^{r}})}\int_{\mathbb{R}^{d}}|f|M_{\varepsilon,L\log L,^{m}}w\\
 & \lesssim\sum_{r=0}^{\infty}\frac{1}{\varepsilon(2^{2^{r}})}\sum_{Q\in\mathcal{S}_{r,k}}\int_{E_{Q}}|f|M_{\varepsilon,L\log L^{m}}w.
\end{align*}
Now we turn our attention to $L_{12}$. Arguing as we did to settle
\ref{eq:JNDecay}, we have by Lemma \ref{lem:LetGenHP} that for $Q\in\mathcal{S}_{r,k}$
\[
\|w\chi_{F_{k}(Q)}\|_{L\log L^{m},Q}\lesssim\frac{k^{m}}{4^{k}}\rho_{m+1,w}(Q)\langle w\rangle_{Q}.
\]
Hence
\begin{align*}
 & \sum_{r=0}^{\infty}\sum_{k=\log_{2}(2^{2^{r+1}})}^{\infty}\sum_{Q\in\mathcal{S}_{r,k}}\frac{1}{|Q|}\int_{Q}|f|\int_{G\cap F_{k}(Q)}|b-b_{Q}|^{m}w\\
 & \lesssim\sum_{r=0}^{\infty}\sum_{k=\log_{2}(2^{2^{r+1}})}^{\infty}\sum_{Q\in\mathcal{S}_{r,k}}\int_{E_{Q}}|f|\|w\chi_{F_{k}(Q)}\|_{L\log L^{m},Q}\\
 & \lesssim\sum_{r=0}^{\infty}\sum_{k=\log_{2}(2^{2^{r+1}})}^{\infty}\sum_{Q\in\mathcal{S}_{r,k}}\int_{E_{Q}}|f|\frac{k^{m}}{4^{k}}\rho_{m+1,w}(Q)\langle w\rangle_{Q}\\
 & \lesssim\sum_{r=0}^{\infty}\sum_{k=\log_{2}(2^{2^{r+1}})}^{\infty}\sum_{Q\in\mathcal{S}_{r,k}}\int_{E_{Q}}|f|\frac{2^{2^{r+1}}k^{m}}{4^{k}}\frac{\varepsilon\left(\rho_{m+1,w}(Q)\right)}{\varepsilon\left(\rho_{m+1,w}(Q)\right)}\langle w\rangle_{Q}\\
 & \lesssim\sum_{r=0}^{\infty}\sum_{k=\log_{2}(2^{2^{r+1}})}^{\infty}\sum_{Q\in\mathcal{S}_{r,k}}\int_{E_{Q}}|f|\frac{2^{2^{r+1}}2^{k}}{4^{k}}\frac{\log_{2}(2+\rho_{m+1,w}(Q))\varepsilon\left(\rho_{m+1,w}(Q)\right)}{\log_{2}(2+\rho_{m+1,w}(Q))\varepsilon\left(\rho_{m+1,w}(Q)\right)}\langle w\rangle_{Q}\\
 & \lesssim\sum_{r=0}^{\infty}\frac{1}{\log_{2}(2+2^{2^{r}})\varepsilon\left(2^{2^{r}}\right)}\sum_{k=\log_{2}(2^{2^{r+1}})}^{\infty}\sum_{Q\in\mathcal{S}_{r,k}}\int_{E_{Q}}|f|M_{\varepsilon,L\log L,m+1}w\frac{2^{2^{r+1}}}{2^{k}}\\
 & \lesssim\sum_{r=0}^{\infty}\frac{2^{2^{r+1}}}{2^{r}\varepsilon\left(2^{2^{r}}\right)}\sum_{k=\log_{2}(2^{2^{r+1}})}^{\infty}\frac{1}{2^{k}}\sum_{Q\in\mathcal{S}_{r,k}}\int_{E_{Q}}|f|M_{\varepsilon,L\log L,m+1}w\\
 & \lesssim\sum_{r=0}^{\infty}\frac{2^{2^{r+1}}}{2^{r}\varepsilon\left(2^{2^{r}}\right)}\sum_{k=\log_{2}(2^{2^{r+1}})}^{\infty}\frac{1}{2^{k}}\int_{\mathbb{R}^{n}}|f|M_{\varepsilon,L\log L,m+1}w\\
 & \lesssim\sum_{r=0}^{\infty}\frac{2^{2^{r+1}}}{2^{r}\varepsilon\left(2^{2^{r}}\right)2^{2^{r+1}}}\int_{\mathbb{R}^{n}}|f|M_{\varepsilon,L\log L,m+1}w\\
 & \simeq\sum_{r=0}^{\infty}\frac{1}{2^{r}\varepsilon\left(2^{2^{r}}\right)}\int_{\mathbb{R}^{n}}|f|M_{\varepsilon,L\log L,m+1}w.
\end{align*}
To provide our estimate for $L_{2},$ we split again in two sums.

\begin{align*}
L_{2} & \leq\sum_{r=0}^{\infty}\sum_{k=1}^{\left\lfloor \frac{r}{2m}\right\rfloor }4^{km}\sum_{Q\in\mathcal{S}_{r,k}}\frac{1}{|Q|}\int_{Q}|f|w(G\cap Q)\\
 & +\sum_{r=0}^{\infty}\sum_{k=\left\lfloor \frac{r}{2m}\right\rfloor }^{\infty}4^{km}\sum_{Q\in\mathcal{S}_{r,k}}\frac{1}{|Q|}\int_{Q}|f|w(G\cap Q)=L_{21}+L_{22}.
\end{align*}
To bound $L_{21}$ we observe that 
\begin{align*}
L_{21}= & \sum_{r=0}^{\infty}\sum_{k=1}^{\left\lfloor \frac{r}{2m}\right\rfloor }4^{km}\sum_{Q\in\mathcal{S}_{r,k}}\frac{1}{|Q|}\int_{Q}|f|w(G\cap Q)\\
 & \lesssim\sum_{r=0}^{\infty}\sum_{k=1}^{\left\lfloor \frac{r}{2m}\right\rfloor }4^{km}\sum_{Q\in\mathcal{S}_{r,k}}\int_{E_{Q}}|f|\frac{w(G\cap Q)}{|Q|}\\
 & \lesssim\sum_{r=0}^{\infty}\sum_{k=1}^{\left\lfloor \frac{r}{2m}\right\rfloor }\sum_{Q\in\mathcal{S}_{r,k}}\frac{\log_{2}\left(2+\rho_{m+1,w}(Q)\right)\varepsilon\left(\rho_{m+1,w}(Q)\right)}{\log_{2}\left(2+\rho_{m+1,w}(Q)\right)\varepsilon\left(\rho_{m+1,w}(Q)\right)}\int_{E_{Q}}|f|\frac{w(G\cap Q)}{|Q|}\\
 & \lesssim\sum_{r=0}^{\infty}\frac{1}{\log_{2}(2+2^{2^{r}})\varepsilon(2^{2^{r}})}\sum_{k=1}^{\left\lfloor \frac{r}{2m}\right\rfloor }4^{km}\sum_{Q\in\mathcal{S}_{r,k}}\int_{E_{Q}}|f|M_{\varepsilon,L,m+1}w\\
 & \lesssim\sum_{r=0}^{\infty}\frac{2^{r}}{\log_{2}(2+2^{2^{r}})\varepsilon(2^{2^{r}})}\int_{\mathbb{R}^{d}}|f|M_{\varepsilon,L,m+1}w\\
 & \lesssim\sum_{r=0}^{\infty}\frac{1}{\varepsilon(2^{2^{r}})}\int_{\mathbb{R}^{d}}|f|M_{\varepsilon,L,m+1}w\\
 & \lesssim\sum_{r=0}^{\infty}\frac{1}{\varepsilon(2^{2^{r}})}\int_{\mathbb{R}^{d}}|f|M_{\varepsilon,L(\log L)^{m},m+1}w
\end{align*}
and hence we are done for this term and it remains to deal with $L_{22}$.
Note that arguing as we did in the previous subsection, for every
cube $Q\in\mathcal{S}_{r,k}$ we have that 
\[
Q=Q^{t}\cup\tilde{E}_{Q}
\]
 where
\[
\sum_{Q\in\mathcal{S}_{r,k}}\chi_{\tilde{E_{Q}}}(x)\leq t
\]
 and
\[
\log\left(2\frac{|Q|}{|Q^{t}|}\right)\geq7^{km}\log\left(\frac{1}{\alpha-1}\right).
\]
Bearing those properties in mind we provide our estimate for $L_{22}.$
We consider the following terms
\begin{align*}
L_{22} & =\sum_{r=0}^{\infty}\sum_{k=\left\lfloor \frac{r}{2m}\right\rfloor }^{\infty}4^{km}\sum_{Q\in\mathcal{S}_{r,k}}\frac{1}{|Q|}\int_{Q}|f|w(G\cap Q^{t})\\
 & +\sum_{r=0}^{\infty}\sum_{k=\left\lfloor \frac{r}{2m}\right\rfloor }^{\infty}4^{km}\sum_{Q\in\mathcal{S}_{r,k}}\frac{1}{|Q|}\int_{Q}|f|w(G\cap\tilde{E}_{Q})\\
 & =L_{221}+L_{222}.
\end{align*}
For $L_{221}$ we observe that 
\begin{align*}
\frac{1}{|Q|}\int_{Q}|f|w(G\cap Q^{t}) & \leq\frac{1}{|Q|}\int_{Q}|f|w(Q^{t})\leq2\|w\|_{L\log L,Q}\|\chi_{Q^{t}}\|_{\exp L,Q}\int_{Q}|f|\\
 & =\frac{1}{\log\left(2\frac{|Q|}{|Q^{t}|}\right)}\int_{Q}f\|w\|_{L\log L}\\
 & \lesssim\frac{1}{7^{km}}\int_{E_{Q}}|f|M_{L\log L}w.
\end{align*}
Hence 
\begin{align*}
\sum_{r=0}^{\infty}\sum_{k=\left\lfloor \frac{r}{2m}\right\rfloor }^{\infty}4^{km}\sum_{Q\in\mathcal{S}_{r,k}}\frac{1}{|Q|}\int_{Q}|f|w(G\cap Q^{t}) & \lesssim\sum_{r=0}^{\infty}\sum_{k=\left\lfloor \frac{r}{2m}\right\rfloor }^{\infty}\frac{4^{km}}{7^{km}}\sum_{Q\in\mathcal{S}_{r,k}}\int_{E_{Q}}|f|M_{L\log L}w\\
 & \lesssim\int_{\mathbb{R}^{d}}|f|M_{L\log L}w\sum_{r=0}^{\infty}\sum_{k=\left\lfloor \frac{r}{2m}\right\rfloor }^{\infty}\left(\frac{4}{7}\right)^{km}\\
 & \lesssim\int_{\mathbb{R}^{d}}|f|M_{L\log L}w\sum_{r=0}^{\infty}\left(\frac{4}{7}\right)^{\frac{r}{2}}\\
 & \lesssim\int_{\mathbb{R}^{d}}|f|M_{L\log L}w.
\end{align*}
Finally, for $L_{222}$, 
\begin{align*}
 & \sum_{r=0}^{\infty}\sum_{k=\left\lfloor \frac{r}{2m}\right\rfloor }^{\infty}4^{mk}\sum_{Q\in\mathcal{S}_{r,k}}\frac{1}{|Q|}\int_{Q}|f|w(G\cap\tilde{E}_{Q})\\
 & =\sum_{r=0}^{\infty}\sum_{k=\left\lfloor \frac{r}{2m}\right\rfloor }^{\infty}4^{mk}\sum_{\nu=0}^{7^{km}}\sum_{Q\in\mathcal{S}_{r,k}}\sum_{Q'\in\mathcal{S}_{r,k}^{\nu}(Q)}\frac{1}{|Q|}\int_{Q}|f|w(G\cap Q')\\
 & \leq\sum_{r=0}^{\infty}\sum_{k=\left\lfloor \frac{r}{2m}\right\rfloor }^{\infty}\frac{4^{mk}}{56^{mk}}\sum_{\nu=0}^{7^{km}}\sum_{Q\in\mathcal{S}_{r,k}}\sum_{Q'\in\mathcal{S}_{r,k}^{\nu}(Q)}w(G\cap Q')\\
 & \leq\sum_{r=0}^{\infty}\sum_{k=\left\lfloor \frac{r}{2m}\right\rfloor }^{\infty}\frac{4^{km}}{56^{km}}7^{km}w(G)\\
 & \leq\sum_{r=0}^{\infty}\sum_{k=\left\lfloor \frac{r}{2m}\right\rfloor }^{\infty}\frac{1}{2^{km}}w(G)\leq8w(G).
\end{align*}
and hence, combining the estimates above we are done.

This ends the proof of Theorem \ref{Thm:TmbS}

\section*{Acknowledgment}

This work will be part of the first author\textquoteright s PhD thesis
at Universidad Nacional del Sur. This research was partially supported
by Agencia I+D+i PICT 2018-02501 and PICT 2019-00018, and by Junta
de Andalucía UMA18FEDERJA002.

\end{document}